\documentclass[12pt,leqno]{article}

\usepackage{amsmath}
\usepackage{amssymb}
\usepackage{amsthm}
\usepackage{mathrsfs}
\usepackage{verbatim} 

\usepackage[T1]{fontenc}
\usepackage{bbm}
\usepackage[english]{babel}
\usepackage{paralist}
\usepackage{fullpage}
\usepackage{color}
\usepackage{stmaryrd}

\usepackage{amsmath}
\usepackage{amsthm}
\usepackage{amssymb}
\usepackage{amsfonts}

\columnwidth\textwidth

\newcommand{\cb}{{\mathcal B}}   
\newcommand{\cc}{{\mathcal C}}

\newcommand{\cf}{{\cal F}}

\newcommand{\ck}{{\mathcal K}}

\newcommand{\cu}{{\mathcal U}}

\newcommand{\R}{\mathbb{R}}   
\newcommand{\N}{\mathbb{N}}  

\newcommand{\Pp}{\mathbb{P}}

\newcommand{\vf}{{\varphi}}

\renewcommand*{\d}{\mathrm{d}}

\newcommand{\ds}{\displaystyle}

\newcommand{\ind}[1]{\mathbbm{1}_{#1}}
\newcommand{\dint}{\mathrm{d}}

\newtheorem{thm}{Theorem}[section]

\newtheorem{prop}[thm]{Proposition}

\newtheorem{rem}[thm]{Remark}  

\numberwithin{equation}{section}

\begin{document}

\noindent
{\Large\bf Stochastic equation of fragmentation
and branching \\ 
processes related to avalanches}\\[2mm]

\noindent 
{\large 
Lucian Beznea\footnote{Simion Stoilow Institute of Mathematics  of the Romanian Academy, Research unit No. 2,  
P.O. Box \mbox{1-764,} RO-014700 Bucharest, Romania, and University of Bucharest, Faculty of Mathematics and Computer Science.  
E-mail: lucian.beznea@imar.ro}, 
Madalina Deaconu\footnote{Inria, Villers-l\` es-Nancy, F-54600, France;
Universit\'e de Lorraine, CNRS, Institut Elie Cartan de Lorraine - UMR 7502, Vandoeuvre-l\`es-Nancy, F-54506, France.  E-mail:  Madalina.Deaconu@inria.fr}, 
and
Oana Lupa\c scu\footnote{ Inria, Villers-l\` es-Nancy;
Universit\'e de Lorraine, CNRS, Institut Elie Cartan de Lorraine--UMR 7502, Vandoeuvre-l\`es-Nancy, France.
{\it Current address:} Institute of Mathematical Statistics and Applied Mathematics of the Romanian Academy,  Calea 13 Septembrie 13, Bucharest, Romania, and the  Research Institute of the University of Bucharest (ICUB).
 E-mail: oana.lupascu@yahoo.com}
}

\vspace{7mm}

\noindent 
{\bf Abstract.} 
We give a stochastic model for the fragmentation phase of a snow avalanche. 
We construct a  fragmentation-branching process related to the avalanches, 
on the set of all fragmentation sizes introduced by J. Bertoin. 
A fractal property of this process is emphasized. 
We also establish a specific stochastic equation of fragmentation. 
It turns out that specific branching Markov processes on finite configurations of  particles 
with sizes bigger than a strictly positive threshold are convenient for describing the continuous time evolution of the number of the resulting fragments. 
The results are obtained by combining analytic and probabilistic potential theoretical  tools.

\vspace{4mm}

\noindent {\bf Key words:} 
Fragmentation kernel, avalanche, branching process, 
stochastic equation of fragmentation, space of fragmentation sizes\\

\noindent
{\bf Mathematics Subject Classification (2010):}
60J80, 
60J45, 
60J35, 
60K35 
\\


\section{Introduction}\label{1}

The snow avalanches were studied from different points of view.
Deterministic avalanche models of Saint-Venant type, including numerical simulations, were investigated, e.g. \cite{PuHu07}, \cite{WHP04}, \cite{GrAn09}, \cite{BouFe08}, \cite{IoLu15}, and \cite{IoLu15a}.

Fractal properties were also emphasized in a natural way, including probabilities.
A fractal model for grain size distribution of a snow avalanche, 
by introducing the concept of aggregation probability, as a coagulation mechanism, is developed in \cite{FaLoGr03}.
The fractal character of the snow has been studied  in the paper \cite{DeBiCh12}. 
In  \cite{CaChFr10} are investigated the scaling properties of snow density in a stochastic fractal framework, 
able to reproduce the local randomness of real microstructure. 

There are several attempts to relay the rock avalanches to some fragmentation  processes; cf. 
\cite{DeBl11}, \cite{DeBl14}, and \cite{GrAn09}, however, the used fragmentation process is deterministic.  
In the articles \cite{GaSaSc93},  \cite{Le04}, and \cite{Za95} a discrete time branching process is associated to an avalanche. 

In this paper we give a stochastic model for the fragmentation phase of a snow avalanche. 
The continuous time evolution of the number of the resulting fragments is described by using specific branching Markov processes.
More precisely, using a recent approach for the fragmentation processes developed in \cite{BeDeLu14}, we construct a 
fragmentation-branching process related to the avalanches, on the set of all fragmentation sizes 
($:=$ the set $S^{\downarrow}$ of all decreasing numerical sequences bounded above from $1$ and with limit $0$;
cf. \cite{Ber06}). 
A fractal property of this process is proved.
We use analytic and probabilistic methods  from the potential theory associated to semigroups and resolvents  of kernels.

The description of our avalanche model is the following.
We fix a rupture factor $r\in (0,1)$, corresponding to an uniform proportionality of the fragments, constant in time.
This specific property of an avalanche will be encoded in our given discontinuous fragmentation kernel.
A Markov chain is induced in a natural way: the state space is $E=[0,1]$, regarded as the space of all fragment sizes, from a point $x\in E$ it is possible to move to the point $\beta x$ and $(1-\beta)x$ with the respective probabilities $\beta$ and $1-\beta$, 
where $\beta: =\frac{r}{1+r}$.
It is an immediate observation that if we assume $\beta<\frac{1}{2}$ then the probability of occurring a bigger fragment is bigger then a smaller one. 
A time continuous Markov chain (actually, a continuous time random walk) is then canonically constructed, it is an analog  of  the Markov process constructed from a fragmentation equation associated to a continuous fragmentation kernel;  see \cite{BeDeLu14} for details and other references.


In the next section we  recall  (for the reader convenience) some properties on the compound Poisson pure jump processes
and we introduce  the discontinuous fragmentation kernels for avalanches we need further to develop  our stochastic model.
We fix a sequence of thresholds $(d_n)_{n\geqslant 1}\subset(0,1)$ strictly decreasing to zero.
In Proposition \ref{prop2.1} we establish compatibility properties between the continuous time Markov chains $X^n$, $n\geqslant 1$,
constructed on each segment $[d_n, 1]$. 
In addition, every process $X^n$  enables to solve a corresponding martingale problem, cf. assertion $(ii)$ of Remark \ref{rem2.1}.
Note that  we cannot associate a stochastic differential equation of fragmentation  with the method from
\cite {fournier-giet} and \cite{BeDeLu14}, because in the present context the fragmentation kernel is discontinuous. 
However, in Section 3 we write down in Theorem \ref{thm3.1} the corresponding  stochastic equation of fragmentation, that is, 
having the  solution equal in distribution with   $X^n$. 
We use essentially the already mentioned existence of the solution of the martingale problem and the technique  from \cite{Jac79} to 
derive a solution for the associated stochastic integro-differential equation, 
which will eventually  lead to the claimed stochastic equation of fragmentation.

As for the case of the fragmentation equation, we are interested in studying the evolution in time of the number of  fragments of the avalanche,  having the same size.  
In Section 4, to handle such a problem,  we first investigate the time evolution of fragments bigger then a given strictly positive threshold $d_n$. 
We consider an adequate kernel controlling the branching  mechanism compatible with the discontinuous fragmentation kernel for the avalanches.  
We  are able to follow the procedure stated in \cite{BeDeLu14}: 
first, we construct  for each $n$ a branching process on the finite configurations of $E$, greater than $d_n$, 
having $X^n$ as base process (cf. Theorem \ref{thm4.1}; for a probabilistic description of this branching process see Remark \ref{rem3.2}), 
and then we project these processes on the set of all fragmentation sizes,  in order to obtain the desired fragmentation process.
We prove in Theorem \ref{thm3.2} that this is actually a branching process with state space $S^{\downarrow}$ 
and that  starting from a sequence of fragmentation sizes $(x_k)_{k\geqslant 1}\in S^{\downarrow}$, 
it is possible to restrict the process to the sequences of sizes of the form 
$\beta^i(1-\beta)^j x_k$, $i,j\in\mathbb{N}$, $k\geqslant 1$.
This emphasizes that an avalanche has a  fractal property, 
which is rather an infinite dimensional one, 
having its origin   in the physical avalanche models studied in the papers mentioned in the first part of the Introduction; 
for more details see assertions $(iii)$ and $(iv)$ of the  Final Remark of the paper.

\section{Fragmentation kernels for avalanches}

\noindent
{\bf First order generators of jump processes.}
Let $(E, \mbox{\sf dist})$ be a metric space, denote by $\cb(E)$ the Borel 
$\sigma$-algebra of $E$, and by $p\cb(E)$ (resp. $b\cb(E)$, $bp\cb(E)$) 
the set of all real-valued positive  (resp. bounded, positive and bounded) $\cb(E)$-measurable functions on $E$. 
We endow $b\cb(E)$ with the sup norm $|| \cdot ||_{\infty}$ and 
denote by $C_l(E)$ the set of all bounded Lipschitz continuous real-valued functions on $E$.\\

Let  $N$ be a kernel on $(E, \cb(E))$, $N\not= 0$.
For all $x\in E$ we denote by $N_{x}$ the measure on $E$  induced by the kernel $N$,  
$$
N_{x}(A):=N (1_A)(x) \mbox{ for all } A\in \cb(E).
$$

Suppose  that
\begin{equation} \label{lipsch}
\int_E \mbox{\sf dist}(x,y) N_x(\mathrm{d}y)<\infty\ \mbox{ for all }\  x\in E
\end{equation}
and  define the {\it first order integral operator} $\widetilde{N}$ induced by $N$ as
$$
\widetilde{N}f(x):= \int_E [f(y)- f(x)] N_x(\mathrm{d}y) \quad \forall f\in bp\cb(E)  \quad \forall  x\in E.
$$
The operator $\widetilde{N}$ is well-defined on the space  $C_l(E)$:
the function $\widetilde{N}f$  is a $\cb(E)$-measurable real-valued function
for every $f\in C_l(E)$.\\

Recall that under various specific conditions (imposed on the kernel $N$, see  \cite{EthKu86}) the linear operator
$\widetilde{N}$ becomes the generator of a Markov process $X=(X_t)_{t\geqslant 0}$ with c\`adl\`ag trajectories,
such that $N$ is the L\'evy kernel of $X$, i.e., 
$$
\mathbb{E}^x \sum_{s\leq t} f(X_{s-}, X_s) = \mathbb{E}^x \int_0^t \int_E f(X_s, y) N_{X_s} (\mathrm{d}y) ds,
$$
for all  $x\in E$ and $f\in bp\cb(E\times E),$  $f=0$ on the diagonal
of $E\times E$.\\

\noindent
{\bf Examples.} We present now several examples of such operators and processes, occurring in
the study of the fragmentation phenomena and of the avalanches.\\

\noindent
{\bf 1. The case of a continuous fragmentation kernel.}
Consider  a {\it fragmentation kernel} $F$, that is,  a symmetric function 
$F: (0, 1]^2\longrightarrow \mathbb{R}_+$, and
recall that $F(x,y)$  represents the rate of fragmentation of a particle of size $x+y$ into two particles of sizes $x$ and $y$.
If $F$ is a continuous function then there exists an associate  stochastic differential equation, as it is  obtained in 
\cite {fournier-giet} and \cite{BeDeLu14}.
More precisely, assume that the  {fragmentation kernel} $F: (0, 1]^2 \longrightarrow \mathbb{R}_+$ 
is a continuous symmetric map. Moreover, $F$ is supposed continuous from $[0,1]^2$ to $\mathbb{R}_+\cup \{+\infty\}$
and  define the function 
$$
\psi(x) =
\left\{
\begin{array}{ll}
\ds\frac{1}{x}\ds\int_0^{x} y(x-y) F(y,x-y) \dint y  & \mbox{ for } x>0,\\[6mm]
0 &  \mbox{ for } x=0,
\end{array}
\right.
$$
which is supposed continuous on $[0, 1]$.  { The function $\psi(x)$ represents the {\it rate of loss of mass of particles of mass x}.} 
With the notations from \cite{fournier-giet} consider the operator $\cf$ defined as
$$
{\cal{F}} f (x) =\ds\int_0^x [f(x-y)- f(x)]\ds\frac{x-y}x{} F(y,x-y)\dint y, \ x\in [0,1]. 
$$
We write $\cf$ as a first order integral operator. 
Let $E=[0,1]$ and consider the kernel
$N^F$ defined as
\begin{equation}\label{ec2.3}  
N^Ff(x):= \int_0^x  f(z) \frac {z}{x} F(x-z, z)\mathrm{d}z,\;\;\; x\in E.
\end{equation}
The condition imposed to the function $\psi$ implies   that $N^F$ satisfies the integrability condition  (\ref{lipsch})
and one can see immediately that $\cf= \widetilde{N^F}$ on $C_l(E)$, that is, 
$$
\cf f(x)= \int_0^1 [f(z)- f(x)] N^F_x (\mathrm{d}z)  \quad  \forall f \in C_l(E)  \quad \forall x\in E.
$$
Observe also that $N^F_0=0$ and 
\begin{equation} \label{2.3}
N^F_x(\mathrm{d}z) =  \frac{z}{x} F(x-z,z)\ind{(0,x)}(z)\mathrm{d}z,\; \mbox{ if } \  x>0.
\end{equation}
By Proposition 2.2 from \cite{BeDeLu14}, using the existence and the uniqueness of the solution 
of the associated stochastic differential equation of fragmentation,
under some mild conditions, $\widetilde{N^F}$ becomes  the generator
of a  $C_0$-semigroup of contractions on $\mathcal{C}(E)$ and consequently 
of a standard (Markov) process  with state space $E$.

In general,  we cannot associate such a stochastic differential equation of fragmentation  because we do not assume that $F$ is continuous. 
However, in Section 3 below we succeed to  write down a stochastic equation of fragmentation, appropriate to the avalanches.
\\

\noindent
{\bf  2. The case of a  bounded kernel $N$.} 
For the reader convenience we present now the classical  situation of a bounded kernel 
(see, e.g.,  \cite{EthKu86}, page 163), as we need to apply it to the fragmentations kernels for avalanches.

Assume that $N1< \infty$ and denote by  $\lambda(x)$ the total mass of the (finite) measure $N_x$, $x\in E$,
$$
\lambda(x):= N1(x)\in\R_+
$$
and consider the induced normalized Markovian kernel $N^o$,
$$
N^o  =\frac{1}{\lambda} N. 
$$
Consequently we have
\begin{equation}\label{ntilda}
\widetilde{N}f(x):= \lambda(x) \int_E [f(y)- f(x)] N^o_x(\mathrm{d}y) \quad \forall f\in bp\cb(E).
\end{equation}

Suppose that $N$ is a bounded kernel.
Then $\widetilde{N}$ becomes a bounded linear operator on $b\cb(E)$,
$$
\widetilde{N}= N-\lambda I, 
$$
and it is the generator of a $C_0$--semigroup $(P_t)_{t\geqslant 0}$ on $b\cb(E)$: 
$$
P_t:=\mathrm{e}^{t\widetilde{N}},\ t\geqslant 0.
$$ 
Each $P_t$  is a Markovian kernel on $E$,
more precisely, if we set
\begin{equation} \label{2.7}
\lambda_o:= ||N1||_{\infty}\ \mbox{ and } \ 
N':= \frac{1}{\lambda_o} N+ (1- \frac{\lambda}{\lambda_o}) I,
\end{equation}
then 
$P_t f=\mathrm{e}^{-t\lambda_o}\sum_{k\geqslant 0} \frac{(\lambda_ot)^k}{k!} N'^kf $, 
where $N'^k:=\underbrace{N' \circ  \ldots  \circ N'}_{k \mbox{ times}}$.
With the above notations we have $\widetilde{N}=\lambda_o \widetilde{N'}$, that is 
  $$
 \widetilde{N}f(x)= \lambda_o \int_E [f(y)- f(x)] N'_x (\mathrm{d}y) \quad\ \forall f\in bp\cb(E)  \quad \forall x\in E.
$$

\noindent
\begin{minipage}[t]{1mm}
{\begin{equation} \label{2.2}\end{equation} }
\end{minipage}

\vspace{-19mm}
 
\hspace*{4mm} 
The operator $\widetilde{N}$ is   the generator of a (continuous time) jump Markov process $(X_t)_{t\geqslant 0}$. 
More precisely,  if $\nu$ is a probability on $E$, we consider the time-homogeneous Markov chain 
$(Y(k))_{k\geqslant 0}$ in $E$, with initial distribution $\nu$ and transition function 
$N'$, that is $\mathbb{P}^{\nu}\circ Y^{-1}(0)=\nu$ and 
$$
\mathbb{E}^{\nu} [ f(Y(k+1)) |Y(0),\ldots,Y(k)]= N'f(Y(k)) \quad\forall \ k\geqslant 0 \quad \forall f \in bp\cb(E).
$$
Let further $(V_t)_{t\geqslant 1}$ be a Poisson process with parameter 
$\lambda_o$ (i.e., $\mathbb{P}^{\nu}(V_t=k)=e^{- t\lambda_o}\frac{(\lambda_ot)^k}{k!}$, for all $k\in\mathbb{N}$ 
and $t\geqslant 0$) which is independent from $(Y(k))_{k\geqslant 0}$. 
Then the process $(X_t)_{t\geqslant 0}$ defined as
$$
X_t:=Y(V_t), \ t\geqslant 0, 
$$
is a c\`adl\`ag Markov process with state space $E$, transition function  $(P_t)_{t\geqslant 0}$ and initial distribution $\nu$.\\

\noindent
{\bf 3. The case of (discontinuous) fragmentation kernels for avalanches.}
Consider  again a  {\it fragmentation kernel} $F: (0, 1]^2\longrightarrow \mathbb{R}_+$.
The following assumption is suggested by the so called {\it  rupture properties}, 
emphasized in the deterministic  modeling of the snow avalanches (cf. the papers mentioned in the Introduction).\\

$(H_1)\quad$ There exists a function $\Phi: (0,\infty) \longrightarrow (0,\infty)$ such that
$$
F(x,y)=\Phi \left( \frac{x}{y} \right) \mbox{ for all } x, y > 0.
$$

Since the fragmentation kernel $F$ is assumed to be a symmetric function, we have
$$
\Phi(z)=\Phi \left(\frac{1}{z}\right)  \mbox{ for all } z>0.
$$

\noindent
{\bf Example of  a fragmentation kernel  satisfying $(H_1)$.}  
Fix a "ratio" $r$, $0<r<1$, and consider the fragmentation kernel $F^r:[0,1]^2\longrightarrow  {\R}_+$, defined as 
$$
F^r(x,y):=
\left\{
\begin{array}{l}
\frac{1}{2}(\delta_r (\frac{x}{y} )+ \delta_{1/r} (\frac{x}{y} ) )\mbox{, if }\; x,y>0,\\[2mm]

0\;\;\;\;\;\;\;\;\;\;\;\;\;\;\;\;\;\;\;\;\;\,\,\;\;\;\; \mbox{, if } \;\;x y=0.
\end{array}
\right. 
$$
We have
$F^r(x,y)=\Phi^{r} (\frac{x}{y})$  for all  $x, y >0,$
where $\Phi^{r}:(0,\infty)\longrightarrow (0,\infty)$ is defined as
$$
\Phi^{r}(z):= \frac{1}{2}(\delta_{r} (z)+ \delta_{1/{r}}(z)), \quad z>0.
$$
Clearly, the function $\Phi^{r}$ is  not continuous.
Let $(\Phi_n^r)_{n}$ be a sequence of continuous functions  
such that  $(\Phi_n^r \!\! \cdot \! \mbox{d}x )_{n}$ is a sequence of probabilities on $(0,\infty)$,
converging weakly to $\delta_r (\mbox{d}x)$.
Let 
$$
\bar{F}_n^r(x,y):=\Phi^r_n \left(\frac{x}{y} \right), \;\; x, y>0.
$$
The function  $\bar{F}_n^r$ is not  symmetric, therefore we consider its symmetrization $F_n^r$, 
$$
F_n^r (x,y):= \frac{1}{2}(\bar{F}_n^r(x,y)+\bar{F}_n^r(y,x)), \;\; x, y>0, 
$$
and let  $N^{F_n^r}$ be the corresponding kernel given by (\ref{ec2.3}). We have by (\ref{2.3})
$$
N^{F_n^r}_x(\mathrm{d}z) =  \frac{z}{x} F_n^r(x-z,z)\ind{(0,x)}(z)\mathrm{d}z, \  x>0.
$$
From the above considerations one can see that for each $x>0$ 
the sequence of measures
$(N^{F_n^r}_x)_n$ converges weakly to $\frac{(1-\beta)^2}{2}[ \beta x\delta_{\beta x}+ (1-\beta)x\delta_{(1-\beta) x}]$,
where $\beta:=\frac{r}{1+r}$.

Since  the sequence  of probabilities $(\Phi_n^{1/r}\!\! \cdot  \mbox{d}x )_{n}$ is approximating
 $\delta_{1/r}  ({\mbox d} x)$ on $(0,\infty)$, it follows  that
$(N^{F_n^{1/r}}_x)_n$ converges weakly to $\frac{\beta^2 }{2}[\beta x \delta_{\beta x}+ (1-\beta)x \delta_{(1-\beta)x}]$.
We conclude that the kernel $N^{F^r}$ associated with $F^r$ is given by the following linear combination of Dirac measures:
 \begin{equation} \label{2.9}   
 N^{F^r}_{x} :=  \lambda_o ( \beta x \delta_{\beta x} + (1-\beta)x \delta_{(1-\beta)x}),
\end{equation} 
where $\lambda_o:= \frac{ \beta^2 + (1-\beta)^2 }{4}$. 
In this case the kernel $N^{F^r}$  is no more Markovian and has no density with respect to the Lebesgue measure.\\

Further on in this paper we take
$$
E:=[0,1].
$$ 

We fix now a sequence of  thresholds for the fragmentation dimensions   $(d_n)_{n\geqslant 1}\subset(0,1)$ strictly decreasing to zero.
For each $n\geqslant 1$ let
$$
E_n:=[d_n,1]\textrm{ and } E'_n:=[d_{n+1},d_n), \;
E'_0 := E_1.
$$ 
We assume that $d_1<\beta$ and $d_{n+1}/d_{n}< \beta$ for all  $n\geqslant 1$.

Let $n\geqslant 1$ be fixed. 
Then  $E_n=\bigcup_{k=1}^n E'_{k-1}$.
The kernel  $N^{F^r}$ from Example  3, given by (\ref{2.9}), is used 
to define  the kernel $N_n^r$  on $E_n$ as
$$
N_n^r f= \sum_{k=1}^n  \ind{E'_{k-1}} N^{F^r} (f\ind{E'_{k-1}} ) \quad \forall f\in bp \cb(E_n).
$$
Observe that, using (\ref{2.9}), 
 \begin{equation} \label{2.10}   
\mbox{ the measure $(N^{r}_n)_x$ is carried by $[d_{k+1}, x]$ for every $x\in E'_k$,} 
\end{equation}
and consider the  corresponding Markovian kernel $N_n'^{r}$, defined as in (\ref{2.7}),
$$
N'^{r}_n f   = \frac{1}{\lambda_o} N^{r}_n f + \left(1- \frac{N^{r}_n 1 }{\lambda_o}\right) f \quad \forall f\in bp \cb(E_n).
$$

Following the procedure from Example 2, we may associate a first order integral operator
$$
\cf_n^r f (x) := \widetilde{N^r_n}f(x) = \int_{E_n}  [f(y)- f(x)]  N^r_n(\mathrm{d}y)  \quad \forall f\in bp\cb(E_n)  \quad\forall x\in E_n.
$$ 
By (\ref{2.2}) the operator $\cf^r_n= \lambda_o (N'^r_n - I)$ is  the generator of a (continuous time) jump Markov process 
$X^n ={(X^n_t)}_{t\geqslant 0}$, 
defined as
\begin{equation} \label{avproc}  
X_t^n:=Y^n(V_t),\quad t\geqslant 0.
\end{equation}
Its transition function is 
$P^n_t:=e^{\cf_n^rt}=e^{-\lambda_o t}\sum_{k{\geqslant}0}\frac{(\lambda_o t)^k}{k!}(N'^r_n)^k$
for all  $t\geqslant 0$ and
let $\cu^{n}= {(U^{n}_\alpha)}_{\alpha >0}$  be the resolvent  associated to the process $X^n$, i.e., 
$U^{n}_\alpha=\int_0^\infty \! e^{-\alpha t} P_t^n\,  \d t,$ $\alpha>0$. 
Recall that $(Y^n(k))_{k\geqslant 0}$ is the time-homogeneous Markov chain with transition function 
$N'^r_n$ and  $(V_t)_{t\geqslant 0}$ is a Poisson process with parameter $\lambda_o$.
It is precisely the random walk described in Introduction.\\

\begin{rem}   \label{rem2.1} 

$(i)$ Consider the integral operator $\cf^r$ defined as
$$
\cf^r f (x) := \widetilde{N^{F^r}}f(x) = \int_{E}  [f(y)- f(x)]  N^{F^r} (\mathrm{d}y)  \quad f\in bp\cb(E) \quad \; x\in E,
$$ 
that is,  we take the kernel $N^{F^r}$ instead of $N^r_n$. 
Then by the procedure described above there exists a  jump Markov process  $X =(X_t)_{t\geqslant 0}$,  
with state space $E$, having $\cf^r$ as generator, $X_t:=Y(V_t),$ $t\geqslant 0$, 
where  $(Y(k))_{k\geqslant 0}$ is the time-homogeneous Markov chain with transition function  $N'^r$,
where the Markovian kernel $N'^r$ is obtained from $N^r$  as in  $(\ref{2.7}).$  
In the next section we show that the process $X$ is related to  a stochastic  equation of fragmentation, 
associated to  $F^r,$ the discontinuous fragmentation kernel for avalanches. 

$(ii)$ 
As in \cite{BeDeLu14}, Proposition 5.2, we are able to solve the martingale problem associated to the bounded operator $\cf^r_n$ (resp. $\cf^r$):
for every $f \in bp\cb(E_n)$   (resp. $f \in bp\cb(E)$)  and each  probability $\nu$ on $E_n$ (resp. $f \in bp\cb(E)$),  the process
$$
f(X^n_t)-   \int_0^t  \cf^r_n f  (X^n_s) \d s,    \quad (\mbox{resp. }   f(X_t)-   \int_0^t  \cf^r  f  (X_s) \d s), \,  \, t{\geqslant}0,
$$
is a martingale under $\Pp^\nu=\int \Pp^x \nu(\d x)$, with respect to the filtration of $X^n$ (resp. of $X$).
 In the Section 4 below $X^n$ (defined by  (\ref{avproc})),  
 will become the base process of a branching Markov process on the finite configurations of $E_n$ (see Theorem \ref{thm4.1} below). 
\end{rem}

For every $x\in[0, 1]$  let 
$$
E_{\beta, x} := \{ \, \beta^i(1-\beta)^j x : \; i, j\in \N \}\cup \{0\}
$$
and for $n\geqslant 1$ let
$$
E_{\beta, x, n}:= E_{\beta, x} \cap E_n. 
$$

\begin{prop}\label{prop2.1}   
 The following assertions hold for the conservative Markov process $X^n$ with state space $E_n$, $n{\geqslant}1$.
 
$(i)$   If $t\geqslant s$ then a.s. $X^n_t \leqslant  X^n_s$. 

$(ii)$ The sets $[d_n, x]$ and $[d_n, x)$ are absorbing subsets of $E_n$ for every $x\in E_n$ with respect to the resolvent $\cu^{n}$ of $X^{n}$. In particular,  $E'_{n-1}$ is an absorbing subset of $E_{n}$. 

$(iii)$ If $x\in E_n$ then $\Pp^x$-a.s. $X^n_t\in E_{\beta, x, n}$  for all $t\geqslant 0$.  

 $(iv)$ 
If $(P_t^n)_{t{\geqslant}0}$  is  the transition function of $X^n$,  then
$$
P^{n+1}_{t,x}=P^{n}_{t,x} \, \textrm{ for all  } t{\geqslant}0  \textrm{ and } x\in E_n.
$$
 \end{prop}

\begin{proof}

$(i)$ Assume that $s=0$ and let $x\in E_n$. 
Then  $X^n_0=x$ $\mathbb{P}^{x}$--a.s., 
so, we have to prove that 
$\mathbb{P}^{x}(X^n_t\in [d_n, x])=1,$
or equivalently,  that $P^n_t (1_{(x,1]})(x)=0$ for all  $t$.
Observe first that if $f=0$ on $[d_n, x]$ then by  (\ref{2.10}), 
$N'^r_n f=0$ on $[0,x]$ and by induction we get $(N'^r_n)^k f=0$ on $[d_n,x]$ for all $k \geqslant 0$. 
Therefore  $P^n_t f= e^{-\lambda_o t}\sum_{k{\geqslant}0}\frac{(\lambda_o t)^k}{k!}(N'^r_n)^k  f$ also vanishes on $[d_n,x]$
for all ${t{\geqslant}0}$. 
The case $s>0$ follows using the Markov property of $X^n$.\\

$(ii)$ We argue as in the proof of Proposition 3.1 from \cite{BeDeLu14}.
By the above considerations we have $\Pp^x$--a.s. $X^n_t\leqslant x$ for all $t\geqslant 0$
and therefore the entry time $D_{(x,1]}$ of the  set $(x,1]$ is infinite. 
Indeed, we have first
$D_{(y,1]}= \infty$  $\Pp^y$--a.s. for all $y\in E_n$ and therefore, 
if $d_n\leqslant y<x$,  then $[D_{(y,1]}= \infty] \subset [D_{(x,1]}= \infty]$,
$\Pp^y([D_{(x,1]})= \infty])=1$. 
By (A.1.2) from  \cite{BeDeLu14} we conclude that $[d_n, x]$ is absorbing.
The set $[0,x)$ is also absorbing as a  union of a sequence of absorbing sets,
$[d_n, x)=\bigcup_k [d_n, d_n+ (x- d_n- \frac{1}{k})^+]$.\\

$(iii)$ It follows by induction from  (\ref{2.9}) that the probability measure $(N'^r_n)^{k}_x$ is carried by 
$E_{\beta, x,n}$ for every $x\in E_n$ and $k\geqslant 0$.
Hence $\Pp^x (X^n_t \in E_{\beta, x, n})=P^n_t(\mathbbm{1}_{ E_{\beta, x, n}})(x)=1$.\\

$(iv)$ It is sufficient to show  that for every $f\in \cb(E_{n+1})$ and $k\geqslant 1$ we have on $E_n$
$(N'^r_{n+1})^k f=(N'^r_{n})^k  (f \mathbbm{1}_{E_n})$.
Indeed, one can see that  
$N^r_{n+1} f= N^r_{n+1} (f \mathbbm{1}_{E_n})=N^r_{n} (f \mathbbm{1}_{E_n})$  on $E_n$
and therefore, by induction, 
$(N'^r_{n+1})^k f= (N'^r_{n+1})^k (f \mathbbm{1}_{E_n})=(N'^r_{n})^k  (f \mathbbm{1}_{E_n})$ on $E_n$.
\end{proof}

If  $X^n$, $n{\geqslant}1$, is the Markov process with state space $E_n$,  given by (\ref{avproc}), 
and having  $(P^n_t)_{t\geqslant 0}$ as transition function,  constructed from the discontinuous fragmentation kernel $F^r$,
then by Proposition \ref{prop2.1},  assertions $(ii)$ and $(iii)$, 
conditions $(H_2)$ and $(H_3)$ from \cite{BeDeLu14}, Section 3, are fulfilled by  the processes $X^n$.

\section{The corresponding stochastic equation of fragmentation} 

To emphasize the stochastic  equation of fragmentation which is related to our stochastic model for the avalanches, 
we rather  consider  for simplicity  the kernel $N^{F^r}$ on $E$  instead of $N^{r}_n$, $n\geqslant 1$, 
and the associated pure jump process $X=(X_t)_{t\geqslant 0}$ with state space $E$,
given in assertion $(i)$ of Remark \ref{rem2.1}.\\[2mm]



We state now the  {\it stochastic   equation of fragmentation for avalanches}:

\begin{equation}\label{stoch-frag-eq}  
X_t= 
X_0 
-  2\lambda_o \beta (1-\beta)  \int_0^t X^2_{\alpha-}  \d \alpha
\end{equation}
$$
-\int_0^t\int_0^1 \int_0^\infty  \left( (1-\beta) X_{\alpha-} \ind{[\frac{s}{\beta\lambda_o} < X_{\alpha-}\leqslant 1]}
+\beta X_{\alpha-}  \ind{[\frac{s}{\lambda_o} < X_{\alpha-}\leqslant  \frac{s}{\beta\lambda_o}]} \right) 
p(\d\alpha, \d u, \d s)
$$
$$
+ \int_0^t  \int_0^\infty  \left( (1-\beta) X_{\alpha-} \ind{[\frac{s}{\beta\lambda_o} < X_{\alpha-}\leqslant 1]}
+\beta X_{\alpha-}  \ind{[\frac{s}{\lambda_o} < X_{\alpha-}\leqslant  \frac{s}{\beta\lambda_o}]}   \right) 
\d\alpha \d s,
$$
where  $p(\d\alpha, \d u, \d s)$ is a Poisson measure with intensity $q:=\d \alpha \d u\d s$.

\vspace{3mm}

The next theorem gives the claimed relation between the equation $(\ref{stoch-frag-eq})$ and the process $X$.

\begin{thm}  \label{thm3.1}   
The  stochastic  equation of fragmentation for avalanches  $(\ref{stoch-frag-eq})$, 
with the initial distribution $\delta_x$, $x\in E$,  has a weak
solution which is equal in distribution with  $(X,\Pp^x)$.
\end{thm}

\begin{proof}
Define the bounded kernel $K$ on $\R$ by

\begin{equation}\label{kernel-K}
K_x= \left\{
\begin{array}{l}
\lambda_o x ( \beta  \delta_{(\beta-1) x} + (1-\beta) \delta_{-\beta x}) \, \mbox{ if  }\,  x\in [0,1], \\[3mm]

0 \, \mbox{ else.}
\end{array}\nonumber
\right.
\end{equation}
Using (\ref{2.9})  one can see that
\begin{equation} \label{3.1a}
\cf^r f(x) = \ds\int_\R [ f(x+y)-f(x) ] K_x( \d y) \mbox{ for all } f\in bp\cb(\R) \mbox{ and } x\in \R,
\end{equation}
where on left hand side the kernel $N^{F^r}$ occurring in the definition of $\cf^r$ is extended from $E$ to $\R$  
with zero on the complement of $E$.
In particular, we have $Kf(x)=\int f(y-x) N^{F^r}_x(\d y)$ for  $x\in[0,1]$.
By a strait-forward  procedure (see e.g., Lemma (2.2) of Chapter V from \cite{BlGe68}), applied to the measure $N^{F^r}_x$ which is carried by $(0,x)$,  we get for all $x, u, s\in \R$:
\begin{equation}\label{K-tau} 
Kf(x)=\int_\R \int_\R \d u \, \d s f(\tau(x, u,s)),
\end{equation}
where
\begin{equation}\label{tau}
\tau(x, u,s)= \left\{
\begin{array}{l}
\inf\{ v>0 : N^{F^r}_x((0, v])>s \} - x \, \mbox{ if  }\,  0\leqslant s, 0\leqslant u\leqslant 1, x\in [0,1], \\[3mm]

0 \, \mbox{ else,}
\end{array}\nonumber
\right.
\end{equation}
with the convention $f(\infty)=0$.

Let further  $(\mathcal{K}, C_b^1(\R))$ be the operator defined as:
$$
\mathcal{K}f(x)=b(x)f'(x)+\int_\R [f(x+y)-f(x)-y\ind{\{|y|\leq 1\}}f'(x)] K_x(\d y),  \; x\in \R;
$$
cf. \cite{Jac79},  page 434.
Taking   $b(x)= \int_\R  y K_x(\d y)=2 \lambda_o\beta(\beta-1) x^2$ 
if $x\in [0,1]$ and $b(x)=0$ elsewhere, we clearly have $\cf^r= \ck$.

We consider  the following stochastic differential equation, 
applying a method from \cite{Jac79}, page 479:  
\begin{equation}\label{sde}  
\d X_t= 2 \lambda_o\beta(\beta-1) X_{t-}^2 \mathrm{d}t +  w_t\d p_t - w_t \ind{[|w_t| \leqslant 1]} \d q_t,
\end{equation}
where
 $ w_t$  is a process that depends on $X_t$, 
such that 
$Kf(X_{t-}) =\int_{\R} \! \int_{\R} \d u\d s f (w_t(\cdot,u,s)) \ind{[w_t(\cdot,u, s)\neq 0]}$.
Note that in general, the existence of the process $w_t$ follows by Theorem  (14.53)  from \cite{Jac79}, applied  to  the measure $\d u\d s$.
However, the main observation here is that  using $(\ref{K-tau})$,
we can take $w_t(\cdot, u, s)= \tau (X_{t-}, u, s)$, or explicitly 
$$
w_t(\cdot, u, s)=
\ind{[0 \leqslant u\leqslant 1]}(u) \ind{[0\leqslant X_{t-} \leqslant 1]} 
\left( (\beta-1) X_{t-} \ind{[\frac{s}{\beta\lambda_o} < X_{t-}]}
-\beta X_{t-}  \ind{[\frac{s}{\lambda_o} < X_{t-}\leqslant  \frac{s}{\beta\lambda_o}]}
+\infty \ind{[X_{t-}\leqslant  \frac{s}{\lambda_o}]}
\right).
$$
From the above considerations, the stochastic differential equation 
$(\ref{sde})$ may be rewritten in the form  (\ref{stoch-frag-eq}).

By  assertion $(ii)$ of Remark \ref{rem2.1} 
we know that the martingale problem associated to the bounded operator $\cf^r$ has a solution. 
Consequently, Theorem (14.80) from \cite{Jac79},  page 48, implies the existence of the solution of the stochastic differential equation $(\ref{sde})$.
\end{proof}

\section{Branching and fragmentation processes of an avalanche}  

For a Borel subset $E$ of $[0,1]$ define 
the {\it space of finite configurations of E} which is
the following set  $\widehat{E}$ of finite positive measures on $E$:
$$
\widehat{E}:=\left\{\sum_{k{\leqslant}k_0}\delta_{x_k}:k_{ 0}\in\N^*, x_k\in E\textrm{ for all } 1{\leqslant}k{\leqslant}k_0\right\}\cup\{{\bf 0}\},
$$
where ${\bf 0}$ denotes the zero measure. 
We identify $\widehat{E}$ with the union of all symmetric $m$-th powers $E^{(m)}$ of $E$:
$\widehat{E}= \bigcup_{m {\geqslant}0}E^{(m)},$ 
where $E^{(0)}:=\{\bf 0\}$; see, e.g.,  \cite{INW68}, \cite{BeOp11}, \cite{BeLu14}, and \cite{Si68}. 
The set $\widehat{E}$ is endowed with the topology of disjoint union of topological spaces 
and the corresponding Borel $\sigma$-algebra $\cb(\widehat{E})$. 


If $p_1, p_2$  are two finite measures on $\widehat{E}$, then
their convolution $p_1 *p_2$ is the finite measure on $\widehat{E}$
defined for every $h\in p{\mathcal B}(\widehat{E})$ by
$$
\int_{\widehat{E}} p_1*p_2(\mbox{d}\nu)h(\nu) := \int_{\widehat{E}} p_1(\mbox{d}\nu_1) \int_{\widehat{E}} p_2(\mbox{d}\nu_2)
h(\nu_1+\nu_2).
$$

If $\varphi\in p\cb(E)$, define the {\it multiplicative function }  $\widehat{\varphi}:\widehat{E}\longrightarrow \R_+$ as
\begin{equation}\nonumber
\widehat{\varphi}({\bf x})=
\left\{
\begin{array}{l}
\displaystyle\prod_{k}\varphi(x_k), \textrm{ if }{\bf x}=(x_k)_{k{\geqslant}1}\in \widehat{E},  {\bf x}\not= {\bf 0},  \\[2mm]

1\;\;\;\;\;\;\;\;\;\;\;\;,  \textrm{ if } {\bf x}= {\bf 0}.
\end{array}
\right. 
\end{equation}

Recall that a bounded kernel $N$  on $\widehat{E}$ is called {\it branching kernel }  if
 $$
N_{\mu+\nu}= N_\mu * N_\nu  \textrm{ for all  } \mu, \nu \in \widehat{E},
$$
where $N_\mu$ denotes  the measure on $\widehat{E}$ such that $\int \! g  \d N_\mu=Ng(\mu)$ for all
$g\in p{\mathcal B}(\widehat{E})$.
Note that if $N$ is a branching kernel on $\widehat{E}$  then $N_{\bf 0}=\delta_{\bf 0}\in M(\widehat{E})$.
A right (Markov) process with state
space $\widehat{E}$ is called {\it branching process}  provided that its transition function is formed by branching kernels.
The probabilistic description of a branching process is the following: if we take two independent
versions $X$ and $X'$  of the process, starting respectively from two measures 
$\mu$ and $\mu'$, then  
$X+X'$  and the process starting from 
$\mu+\mu'$ 
are equal in distribution.\\

\noindent
{\bf Branching processes on the finite configurations of $E_n$}.
For all $n\geqslant 1$ we define the Markovian kernel $B^n$ from $\widehat{E_n}$ to $E_n$ as
\begin{equation}\label{branker}
B^nh(x) := \frac{1}{a(x)}\sum_{{1}\leqslant k\leqslant n}\, \sum_{E_{\beta, x}\ni y \leq x}\!\!\! \!\!\!
\mathbbm{1}_{E'_{k-1}}\! (x) {}_{d_k}\! h(y , y ) y(x-y),\; 
h\in bp\cb(E_n),\, x\in E_n, 
\end{equation}
where for $d> d_n$ and $g \in bp{\mathcal B}([d,1]) $ we consider  the function $ {}_dg \in bp{\mathcal B}(E_n) $, 
the extension  of $g$   to $E_n$ with the value $g(d)$ on $[d_n, d)$,
$$
{}_dg(y):=g(d) \mathbbm{1}_{[d_n,d)}(y)+g(y)\mathbbm{1}_{[d,1]}(y), \; y\in E_n,
$$
and $a(x)\! :=\!\!\!\!\!\!  \displaystyle\sum_{E_{\beta, x}\ni y \leq x}  \!\!\!\!\! y(x-y )< \infty$ for all $x\in E_n$.


Observe that $\sup_{x\in E_n} B^n l_1(x)=2$, where for a function $f\in p\cb(E_n)$ we consider   the mapping  
$l_f: \widehat{E_n}\longrightarrow {{\R}}_+$   defined as
$ l_f(\mu)
:=\int f \mbox{d}\mu, \  \mu\in \widehat{E_n}.$ 
Therefore we my apply Proposition 4.1 from \cite{BeLu14}, 
to construct a transition function $(\widehat{P^n_t})_{t\geqslant 0}$ on $\widehat{E_n}$,
induced by  $(P_t^n)_{t{\geqslant}0}$, 
the transition function of the Markov process $X^n$ on $E_n$, given by (\ref{avproc}), and by the kernel $B^n$.
The existence of a branching process with state space $\widehat{E_n}$ and  transition function $(\widehat{P^n_t})_{t\geqslant 0}$ 
is given by the first assertion of the following theorem. 
It is a version of Proposition 5.1 from \cite{BeDeLu14}.

If 
$x_1, \ldots , x_k \in E$ and  ${\bf x}=\delta_{x_1}+\ldots +\delta_{x_k}\in \widehat{E}$, we put
$$
E_{\beta, {\bf x}}:= \bigcup_{j=1}^{k} E_{\beta, x_j}.
$$
If $n\geqslant 1$ then let
$$
E_{\beta, {\bf x},n}:= \bigcup_{j=1}^{k} E_{\beta, x_j, n}.
$$

\begin{thm}  \label{thm4.1}  
Let $n\geqslant 1$. Then the following assertions hold.

$(i)$ There exists a branching standard (Markov) process $\widehat{X^n}=(\widehat{X^n_t})_{t \geqslant 0}$, 
induced by $(P_t^n)_{t{\geqslant}0}$ and the kernel 
${B^n}$, with state space $\widehat{E_n}$ and having the transition function $(\widehat{P^n_t})_{t{\geqslant}0}$.

$(ii)$  For every ${\bf x}\in \widehat{E_n}$,  ${\bf y}\in \widehat{E_{\beta, {\bf x}, n}}$ , and $t\geqslant 0$ we have $\Pp^{\bf y}$--a.s. $\widehat{X^n_t}\in \widehat{E_{\beta, {\bf x}, n}}$.
\end{thm}

\begin{proof} 
$(i)$ We argue as in the proof of Proposition 5.1 from \cite{BeDeLu14}. 
Consider the vector space $\cc_n$ defined as
$$
\cc_n :=\{ f:[d_n,1]\longrightarrow \R\, :\, f|_{E'_k}\in\cc(E'_k)\textrm{ such that }\lim_{y\nearrow d_k}f(y)\in\R \mbox{ for all } \, k=0,\ldots,n-1\}.
$$
We claim that  $(P_t^n)_{t{\geqslant}0}$ induces  a $C_0$-semigroup of contractions on $\cc_n$. 
Indeed, we observe first that the Markovian kernel $N'^r_{n}$ becomes a bounded contraction operator on $\mathcal{C}(E_n)$ and consequently 
$P^n_t(\cc_n)\subset\cc_n$ for all $t{\geqslant}0$. 
Also, for all $f\in\mathcal{C}_n$ we have $\lim_{t\searrow 0}P^n_t f=f$ uniformly on $E'_k$ for each $k=0,\ldots,n-1.$ 
Note that the kernels $B^n$  have the following property
$$
\mbox{ if $\varphi\in p\cb(E_n),\varphi{\leqslant}1,$ then $B^n \widehat{\varphi}\in\cc_n$.}
$$ 
Now  condition $(*)$ from the proof of Proposition 5.1 from \cite{BeDeLu14} is verified by the vector space $\cc_n$. 
So,  there exists a standard process with state space $\widehat{E_n}$, 
having $(\widehat{P^n_t})_{t{\geqslant}0}$ as transition function.

$(ii)$ Let $\vf:=\mathbbm{1}_{E_{\beta, {\bf x}, n}}$ and observe that $\widehat{\vf}=\mathbbm{1}_{\widehat{E_{\beta, {\bf x},n}}}$.
Since for ${\bf y}\in\widehat{ E_{\beta, {\bf x},n} }$ we have 
$\widehat{E_{\beta, {\bf y},n}} \subset \widehat{E_{\beta, {\bf x},n}}$, 
we only have to prove that $\widehat{P^n_t}\widehat{\vf}({\bf x})=1$.
By Proposition 4.1 from \cite{BeLu14} we have 
$\widehat{P_{t}^{n}} \widehat{\varphi} =\widehat{h_t(\vf)}$,
where $h_t(\vf)=:h_t\in bp\cb(E_n)$
is the unique solution of the integral equation on $E_n$
$$
{h_t(x)= e^{-t}P_t^{n}\varphi (x)+ \int_{0}^{t} e^{-(t-u)} P^{n}_{t-u}
B^{n}\widehat{h_u} (x) \d u,\; t\geqslant  0,\; x\in E_{n}.}
$$
It is sufficient to show that 
\begin{equation} \label{2.14}
h_t(\mathbbm{1}_{E_{\beta, {x}, n}})=1 \mbox{  on } E_{\beta, {x}, n} \mbox{ for every } x\in E_n. 
\end{equation}
Indeed, if ${\bf x}\in \widehat{E_n}$, ${\bf x}=\delta_{x_1}+\ldots +\delta_{x_k}$, 
since the function $\vf \longmapsto h_t(\vf)(x)$ is increasing,  we have
$$
\widehat{P^n_t}\widehat{\vf}({\bf x})=
\prod_{j=1}^k h_t(\vf)(x_j) \geqslant
\prod_{j=1}^k h_t(\mathbbm{1}_{ E_{\beta, x_j, n}})(x_j)=1,
$$
where the last equality holds by  (\ref{2.14}). 
To prove it let $x\in E_n$ and assume further that $\vf:=\mathbbm{1}_{E_{\beta, {x},n}}$.
Recall that (cf. the proof of Proposition 4.1 from \cite{BeLu14}) $h_t$ is the point-wise limit of the sequence
$(h^m_t)_{m\in \N}$ defined inductively as follows: 
$h^0_t         = e^{-t}P^n_t \vf + \int_{0}^{t} e^{-(t-u)} P^{n}_{t-u} B^{n}\widehat{\vf} \d u$,
$h^{m+1}_t = e^{-t}P^n_t \vf + \int_{0}^{t} e^{-(t-u)} P^{n}_{t-u} B^{n}\widehat{h^m_u} \d u$, $m\geqslant 0$.
Assertion $(iii)$ of Proposition \ref{prop2.1} implies that
$P^n_t \vf= 1$ on $E_{\beta, x, n}$ for all $t\geqslant 0$ and by (\ref{branker}) we also have 
$B^n\widehat{\vf}=1$ on $E_{\beta, x,n}$. 
It follows that on $E_{\beta, x, n}$ we have $h^0_t =1$  and by induction $h^m_t =1$ for every $m\geqslant 0$.
We conclude that (\ref{2.14}) holds.
\end{proof}

\begin{rem} \label{rem3.2} 
$(i)$ The branching process constructed in Theorem \ref{thm4.1} on  $E_n$ has the following description: 
{\rm An initial particle starts at a point of $E_n$ and
moves according to the {\it base process} $X^n$,  until a random time, when
it splits randomly into two new particles, its direct descendants, placed in $E_n$. 
Each direct descendant starts from a position  dictated by the "branching kernel" 
$B^n$  and moves on according to  two independent copies of $X^n$ and so on. 
}

$(ii)$ In the above description the number of direct descendants is two since the kernel $B^n$ from $\widehat{E_n}$ to $E_n$, defined by
(\ref{branker}),  is actually carried by $E_n^{(2)}$.
Accordingly, a binary fragmentation process will be involved.

$(iii)$ In contrast with the situation from  \cite{Si68},  the starting position of a descendant may be different from the terminal position of the parent particle. 
This property is compatible with our aim to count with the branching processes the  number of  fragments of the avalanche,  having the same size. 
So, since for our model a particle represents  a size, then clearly the descendants, resulting after a splitting, 
should represent sizes which are smaller than  (in particular, different from) the parent size.
\end{rem}

\noindent
{\bf Branching processes on the space of fragmentation sizes.} 
As  in \cite{BeDeLu14}, we intend to construct a process with state space the set $S^{\downarrow}$ 
of all decreasing numerical sequences bounded above from $1$ and with limit $0$,

$$
{ S^{\downarrow}}:=\{ {\bf x}=(x_k)_{k{\geqslant}1}\subseteq [0,1]:(x_k)_{k{\geqslant}1}\textrm{ decreasing}, \lim_{k}x_k=0 \}.
$$
Recall  that a  sequence  $ {\bf x}$ from $S^{\downarrow}$ may be considered as
 "the sizes of the fragments resulting from the split of some block with unit size" 
 (cf. \cite{Ber06}, page 16). 
It is convenient to identify a sequence ${\bf x}=(x_k)_{k{\geqslant}1}$ from $S^{\downarrow}$ 
with the $\sigma$-finite measure $\mu_{{\bf x}}$  on $[0,1]$, defined as

$$ 
\mu_{{\bf x}}:=
\left\{
\begin{array}{l}
\displaystyle\sum_{k}\delta_{x_k}\;\;, \textrm{ if }{\bf x}\neq {\bf 0},\\[2mm]

\;\;\;{\bf 0}\;\;\;\;\;\; \;\, ,\textrm{ if } {\bf x}= {\bf 0},
\end{array}
\right. 
$$
where the zero constant sequence ${\bf 0}$  is identified with the zero measure,  $\mu_{\bf 0}=\bf 0$.

 Let further
$$
{ S}:=\{ {\bf x}=(x_k)_{k{\geqslant}1}\in S^{\downarrow}:  \exists\, k_0\in \mathbb{N}^* \textrm{ such that } \,
x_{k_0}>0 \mbox{ and }  x_k=0 \textrm{ for all } k>k_0\}.
$$
 
The mapping ${\bf x}\longmapsto \mu_{{\bf x}}$ 
identifies $S$ with $\bigcup_{n{\geqslant}1} \widehat{E_n}\setminus\{\bf 0\}$. 
For ${\bf x} \in S^{\downarrow}$ we write ${\bf x}=\mu_{{\bf x}}$
where it is necessary to emphasize the identification of the sequence ${\bf x}$ with the measure $\mu_{{\bf x}}$.

In order to consider branching kernels on $S^{\downarrow}$ we need to have a convolution operation between finite measures on $S^{\downarrow}$. 
We first endow $S^{\downarrow}$ with a semigroup structure: if 
$ {\bf x},  {\bf y} \in S^{\downarrow}$ then 
the sequence ${\bf x}+ {\bf y} \in S^{\downarrow}$ is by definition  the decreasing rearrangement 
of the terms of the sequences ${\bf x}$ and ${\bf y}$. 
The convolution may be now introduced as in the case of the space finite configurations:
if $p_1, p_2$  are two finite measures on $S^{\downarrow}$, then
their convolution $p_1 *p_2$ is the finite measure on $S^{\downarrow}$,
defined for every $h\in p{\mathcal B}(S^{\downarrow})$ by:
$\int_{S^{\downarrow} } p_1*p_2(\mbox{d}\nu)h(\nu) :=$ 
$ \int_{S^{\downarrow} } p_1(\mbox{d}\nu_1) \int_{S^{\downarrow} } p_2(\mbox{d}\nu_2) h(\nu_1+\nu_2).$ 
The branching kernels on $S^{\downarrow}$ and the branching process with state space 
$S^{\downarrow}$ are now defined analogously, and the probabilistic interpretation remains valid. 

Define the mapping $\alpha _n: S^{\downarrow}\longmapsto \widehat{E_n}$ as
$\alpha_n({\bf x}):=\mu_{{\bf x}}|_{E_n},\; {\bf x}=\mu_{{\bf x}}\in S^{\downarrow}.$
We have $\alpha_n({\bf 0})={\bf 0}$ and
$\alpha_n|_{\widehat{E_n}}=\textrm{Id}_{\widehat{E_n}}.$
Define also
$$
S_{\infty}:=\{({\bf x}^n)_{n{\geqslant}1}\in \prod_{n{\geqslant}1}\widehat{E_n}: \, {\bf x}^n=\alpha_n({\bf x}^m) \textrm{ for all }m>n{\geqslant}1\}.
$$
By Proposition 4.5 from \cite{BeDeLu14} we have:

\vspace{-3mm}

\noindent
\begin{minipage}[t]{1mm}
{\begin{equation} \label{3.4} \end{equation} }  
\end{minipage}

\vspace{-18.3mm}
 
\hspace*{4mm}   
The mapping $i:S^{\downarrow}\longrightarrow S_{\infty} $, defined as
$i({\bf x}):=(\alpha_n({\bf x}))_{n{\geqslant}1},\; {\bf x}\in S^{\downarrow},$
is a bijection.

\vspace{3mm}

Using assertion $(iv)$ of Proposition \ref{prop2.1}, we may apply Proposition 4.6 from  \cite{BeDeLu14} to get the following result:

\vspace{-3mm}

\noindent
\begin{minipage}[t]{1mm}
{\begin{equation} \label{3.5} \end{equation} }  
\end{minipage}

\vspace{-19.1mm}
\hspace*{4mm}
Let ${\bf x}\in S^{\downarrow}$ and ${\bf x}_n:=\alpha_n({\bf x})\in\widehat{E_n},n{\geqslant}1.$ If $t>0$ 
then the sequence of probability measures $(\widehat{P}_{t,{\bf x}_n})_{n{\geqslant}1}$ 
is projective with respect to $(\widehat{E_n},\alpha_n)_{n{\geqslant}1}$, that is 
$\widehat{P^{n+1}_{t,{\bf x}_{n+1}}}\circ \alpha^{-1}_n=\widehat{P^{n}_{t,{\bf x}_{n}}}$   for all  $n{\geqslant}1.$

\vspace{3mm}

By the identification of $S^{\downarrow}$ with $S_{\infty}$, given in (\ref{3.4}), 
the projective system of probabilities from (\ref{3.5}) may be used to apply Bochner-Kolmogorov Theorem
(see, e.g., \cite{BeCi14}) in order to induce  a transition function on the space of all fragmentation sizes, 
as in Proposition 4.7 from \cite{BeDeLu14}:

\noindent
\begin{minipage}[t]{1mm}
{\begin{equation} \label{3.6} \end{equation} }  
\end{minipage}

\vspace{-18.7mm}
\hspace*{4mm}
There exists a Markovian transition function 
$(\widehat{P_t})_{t{\geqslant}0}$ on $S^{\downarrow}$ 
such that for each ${\bf x}\in  S^{\downarrow}$ and $n{\geqslant}1$ we have
$\widehat{P_{t,{\bf x}}}\circ \alpha^{-1}_n=\widehat{P^{n}_{t,{\bf x}_{n}}},$ 
where ${\bf x}_{n}:=\alpha_n({\bf x})$.

\vspace{3mm}

We can state now the result on the fragmentation-branching processes related to the avalanches, 
having as state space the set $S^{\downarrow}$ of al  fragmentation sizes, 
endowed with the topology induced by the identification  from (\ref{3.4}) with $S_\infty$  
(equipped with the product topology).

\begin{thm} \label{thm3.2}  
The following assertions hold.

$(i)$ There exists a branching right (Markov) process $\widehat{X}= (\widehat{X_t})_{t{\geqslant}0}$ with state space $S^{\downarrow}$, having c\` adl\` ag trajectories 
and   the transition function  $(\widehat{P_t})_{t{\geqslant}0}$,  given by (\ref{3.6}).

$(ii)$ If  $ x\in[0,1]$, $\mbox {\bf y}\in S^{\downarrow}$, 
$\mbox{\bf y}{\leqslant}x,$  and $t{\geqslant}0$ we have
$\widehat{X_t} {\leqslant}x$  $\Pp^{\mbox{\small \bf y}}$--a.s.,
where if $\mbox{\bf y}=(y_k)_{k{\geqslant}1}$, ${\bf y}{\leqslant}x$ means that $y_k{\leqslant}x$ for all $k{\geqslant}1$.

$(iii)$ For each $ {\bf x}\in\widehat{E} $, the set $S^{\downarrow}_{\beta, {\bf x}}:= 
\{ {\bf y}=(y_k)_{k{\geqslant}1}\in S^{\downarrow} :  y_k \in E_{\beta, {\bf x}}$ for all  $k{\geqslant}1 \}$ 
is absorbing in
$S^{\downarrow}$, that is, if $\mbox {\bf y}\in S^{\downarrow}_{\beta, {\bf x}}$ then $\Pp^{\mbox{\small \bf y}}$--a.s.
$\widehat{X_t}\in S^{\downarrow}_{\beta, {\bf x}}$ for all $t\geqslant 0$.
\end{thm}

\begin{proof}
We prove that each kernel $\widehat{P_t}$, $t\geqslant 0$, is a branching kernel on $S^{\downarrow}$.
Let $ {\bf x},  {\bf y} \in S^{\downarrow}$, we have to show that
$$
\widehat{P_{t, {\bf x}+{\bf y}}}(g)=\widehat{P_{t, {\bf x}}}*\widehat{P_{t, {\bf y}}}(g) \mbox{ for all } g\in bp\cb(S^{\downarrow}).
$$
Since by (5.10) from \cite{BeDeLu14} we have 
$\cb(S^{\downarrow})=\sigma\left(\bigcup_{n{\geqslant}1}\{f\circ\alpha_n\;:\;f\in bp\cb(\widehat{E_n})\}\right)$,
we may assume that $g=f\circ\alpha_n$ for some $n\geqslant 1$, with $f\in bp\cb(\widehat{E_n})$
and using  (\ref{3.6}) it remains  to prove that
\begin{equation} \label{3.7}
\widehat{P^n_{t, {\bf x}_n+{\bf y}_n}}(f)=\widehat{P^n_{t, {\bf x}_n}}*\widehat{P^n_{t, {\bf y}_n}}(f) \mbox{ for all } f\in bp\cb(\widehat{E_n}),
\end{equation}
where ${\bf x}_{n}:=\alpha_n({\bf x})$. 
Here we used the equality $\alpha_n({\bf x}+ {\bf y})=\alpha_n({\bf x})+\alpha_n({\bf y})$, where the second sum
is the usual addition of measures from $\widehat{E_n}$.
But the equality (\ref{3.7})  is precisely the branching property of the kernel $\widehat{P_t^n}$ on $E_n$, which clearly holds 
by Theorem \ref{thm4.1}.

The existence of the process $\widehat{X}$ claimed in assertion $(i)$ and assertion $(ii)$ 
are consequences of \cite{BeDeLu14}, Theorem 5.3 and Corollary 5.5 respectively.
Note that  a key argument in the proof of the path regularity is the existence of a compact Lyapunov function 
(a superharmonic function having compact level sets); 
see \cite{BeLuOp11} and also \cite{BeRo11}  for some related results.

$(iii)$ By the right continuity of the process $\widehat{X}$ and since the set $S^{\downarrow}_{\beta, {\bf x}}$ is closed, it is sufficient to
show that for each $t\geqslant 0$ we have $\widehat{P_t}(\mathbbm{1}_{S^{\downarrow}_{\beta, {\bf x}}})({\bf y})  =1$ for every 
${\bf y}\in S^{\downarrow}_{\beta, {\bf x}}$.
We have: 
${\bf y}\in S^{\downarrow}_{\beta, {\bf x}}$ if and only if $\alpha_n({\bf y})\in \widehat{E_{\beta, {\bf x},n}}$ for all $n\geqslant 1$.
On the other hand by assertion $(ii)$ of Theorem \ref{thm4.1} we get $\widehat{P_t^n}(\mathbbm{1}_{\widehat{E_{\beta, {\bf x}, n}}})({\bf y}_n)=1$
provided that ${\bf y}_n$ belongs to $\widehat{E_{\beta, {\bf x}, n}}$.
Using again (\ref{3.6}), we conclude that
$$
\widehat{P_t}(\mathbbm{1}_{S^{\downarrow}_{\beta, {\bf x}}})({\bf y})=
\widehat{P_{t, {\bf y}}}(\bigcap_{n\geqslant 1} \alpha_n^{-1}(\widehat{E_{\beta, {\bf x}, n}}))=
\lim_n \widehat{P_{t, {\bf y}}}( \alpha_n^{-1}(\widehat{E_{\beta, {\bf x}, n} }))=
\lim_n  \widehat{P^n_{t, {\bf y}_n}}( \widehat{E_{\beta, {\bf x}, n} })=1,
$$
where ${\bf y}_n=\alpha_n({\bf y}) \in \widehat{E_{\beta, {\bf x}, n} }.$
\end{proof}

\noindent
{\bf Final remark.}
$(i)$ Recall that at the origin of the equation $(\ref{stoch-frag-eq})$,
the stochastic   equation of fragmentation for avalanches,
 is the discontinuous fragmentation kernel $F^r$.
This equation should be compared with the stochastic
differential equations of fragmentation with continuous fragmentation kernels from \cite{fournier-giet} and
equation (2.4) from \cite{BeDeLu14}.

$(ii)$ The statement of assertion $(i)$ of Theorem \ref{thm3.2} is valid, with the same proof, for the fragmentation process 
constructed starting with a fragmentation equation (with a continuous fragmentation kernel). 
More precisely, the right (Markov) process, given by Theorem 5.3 from \cite{BeDeLu14}, 
is in addition  a branching process with state space $S^{\downarrow}$.
This  is similar to the branching property  proved for the fragmentation chains in  Proposition 2.1  from \cite{Ber06}.

$(iii)$ Since by assertion $(iii)$ of Theorem \ref{thm3.2} the set $S^{\downarrow}_{\beta, {\bf x}}$ is absorbing, it is possible to restrict 
the fragmentation-branching process $\widehat{X}$ to this set; 
for the restriction procedure see, e.g., $(12.30)$  in \cite{Sh88} and Appendix A.1 in \cite{BeDeLu14}.

$(iv)$ Assertion $(iii)$ of Theorem \ref{thm3.2} emphasizes the fractal property of an avalanche, claimed in the Introduction and
closed to its real physical  properties: 
if we regard the fragmentation--branching process on the set $S^{\downarrow}_{\beta, {\bf x}}$
(which is possible by restriction, according to the previous assertion $(iii)$),
then independent to the sequence of sizes ${\bf x}$ of the initial fragments, from the moment when the avalanche started, and  remaining constant in time,
the ratio  between the resulting fragments are all powers of $\beta$.

$(v)$ In a future paper we intend to give a probabilistic numerical approach to the avalanches, based on the results from this paper,
and using appropriate stochastic equations of  fragmentation.\\

\noindent
{\bf Acknowledgements.} The authors acknowledge enlightening discussions with Ioan R. Ionescu, during the preparation of this paper.
This work was supported by a grant of the Romanian National Authority for Scientific Research, CNCS-UEFISCDI, project number PN-II-ID-PCE-2011-3-0045.  For the third author the research was financed  through the project "Excellence Research Fellowships for Young Researchers", the 2015 Competition, founded by the Research Institute of the University of Bucharest (ICUB).

\end{document}